\documentclass{amsart}
\usepackage{amssymb}
\usepackage{amsthm}
\newtheorem{theor}{Theorem}[section] 
\newtheorem{prop}{Proposition}
\theoremstyle{definition} \newtheorem{defin}{Definition}[section]
\newtheorem{ex}{Example}[section]
\theoremstyle{remark} \newtheorem{rem}{Remark}[section]
\newcommand{\pn}{\par\noindent} \newcommand{\pmn}{\par\medskip\noindent}

\begin{document}
\title{Polygons in three-dimensional space}
\author{Yury Kochetkov}
\date{}
\begin{abstract} Let $P=A_1\ldots A_n$ be a generic polygon in
three-dimensional space and let $v_1,v_2,\ldots,v_n$ be vectors
$\overline{A_1A_2},\overline{A_2A_3},\ldots,\overline{A_nA_1}$,
respectively. $P$ will be called \emph{regular}, if there exist
vectors $u_1,\ldots,u_n$ such that cross products
$[u_1,u_2],[u_2,u_3],\ldots,[u_n,u_1]$ are equal to vectors
$v_2,v_3,\ldots,v_1$, respectively. In this case the polygon $P'$,
defined be vectors $u_2-u_1,u_3-u_2,\ldots,u_1-u_n$ will be called
the \emph{derived polygon} or the \emph{derivative} of the polygon
$P$. In this work we formulate conditions for regularity and
discuss geometric properties of derived polygons for $n=4,5,6$.
\end{abstract}
\email{yukochetkov@hse.ru, yuyukochetkov@gmail.com}
\maketitle

\section{Introduction}
\pn In this work we consider duality problems in the
three-dimensional space. The duality for plane polygons is
discussed in work \cite{CN} (in particular, the duality of plane
quadrangles is the subject of work \cite{BK}). However, the use of
complex numbers is the main tool in this approach. In three
dimensional space the natural "multiplication"{} is the cross
product. So we try to construct dual objects for space polygons
using cross product, as main tool. Spaces of three-dimensional
polygons are considered in \cite{HK} from points of view of
differential and algebraic geometry. However, in this work we use
only elementary properties of cross and dot products in
three-dimensional space, so it can be understand by
undergraduates. The standard reference here is \cite{SL}. \pmn We
work in the standard space $\mathbb{R}^3$. In what follows $(a,b)$
will be the dot product of vectors $a$ and $b$, and $[a,b]$ will
be the cross product. Let us remind that
$$(a,[b,c])=\begin{vmatrix}\,a_1&a_2&a_3\\\,b_1&b_2&b_3\\\,c_1&c_2&c_3
\end{vmatrix}$$ for any three vectors $a=(a_1,a_2,a_3)$, $b=(b_1,b_2,b_3)$
and $c=(c_1,c_2,c_3)$. \pmn
Let $P=A_1,\ldots,A_n$ be a generic polygon in three-dimensional space, i.e.
any pair of its consecutive edges are not collinear and any triple of its
consecutive edges are not coplanar. Let $v_1=\overline{A_1A_2},
v_2= \overline{A_2A_3},\ldots,v_n=\overline{A_nA_1}$. We want to
construct a system
of vectors $u_1,\ldots,u_n$ such, that
$$[u_1,u_2]=v_2,\,\,[u_2,u_3]=v_3,\,\ldots,\,[u_n,u_1]=v_1.$$
\begin{defin} A generic polygon $P$ is called \emph{regular}, if such system
exists and the system itself will be called a \emph{support
system} for the polygon $P$. If $P$ is regular and
$u_1=\overline{OB_1},\ldots, u_n=\overline{OB_n}$ --- its support
system, then the polygon $P'=B_1\ldots B_n$ will be called the
\emph{derived polygon} (or the \emph{derivative}) of $P$.
\end{defin}
\begin{ex} Let $P=A_1A_2A_3A_4$, where $A_1=(0,0,0), A_2=(1,1,2),
A_3=(2,3,1), A_4=(-1,2,-2)$, then $v_1=(1,1,2)$, $v_2=(1,2,-1)$,
$v_3=(-3,-1,-3)$ and $v_4=(1,-2,2)$. As $u_1\bot v_1$ and $u_1\bot
v_2$, then vector $u_1$ is a multiple of the cross product
$[v_1,v_2]$. Analogously, $u_2$ is a multiple of $[v_2,v_3]$,
$u_3$ is a multiple of $[v_3,v_4]$ and $u_4$ is a multiple of
$[v_4,v_1]$.\pmn Let $u_1=[v_1,v_2]=(-5,3,1)$. As
$[u_1,[v_2,v_3]]=(9,18,-9)$, than $u_2=\frac 19\cdot
[v_2,v_3]=\left(-\frac 79,\frac 23,\frac 59\right)$. Analogously,
as $[u_2,[v_3,v_4]]=(3,1,3)$, then $u_3=-[v_3,v_4]=(8,-3,-7)$. And
as $[u_3,[v_4,v_1]]=(-9,18,-18)$, then $u_4=-\frac 19\cdot
[v_4,v_1]= (\frac 23,0,-\frac 13)$. At last we have, that
$$[u_4,u_1]=\left(\left(\frac 23,0,-\frac 13\right),(-5,3,1)\right)=
(1,1,2)$$
--- support system is constructed and our polygon is regular.
\pmn Now let us study the derived polygon $P'=B_1B_2B_3B_4$. As
$$(u_2-u_1,[u_3-u_1,u_4-u_1])=\begin{vmatrix}38/9&-21/9&-4/9\\ 13&-6&-8\\
17/3&-3&-4/3\end{vmatrix}=0,$$ then the polygon $P'$ is a plane
quadrangle. Moreover, as
$$[u_2-u_1,u_3-u_1]+[u_3-u_1,u_4-u_1]=(16,28,5)+(-16,-28,-5)=(0,0,0),$$
then the oriented area of $B_1B_2B_3B_4$ is zero. Thus, $P'$ is
self-intersecting:
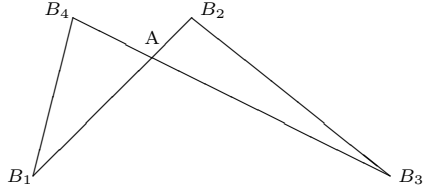
\[\begin{picture}(155,80) \put(10,10){\line(1,1){60}}
\put(10,10){\line(1,4){15}} \put(25,70){\line(2,-1){120}}
\put(70,70){\line(5,-4){75}} \put(0,8){\scriptsize $B_1$}
\put(73,71){\scriptsize $B_2$} \put(148,8){\scriptsize $B_3$}
\put(14,71){\scriptsize $B_4$} \put(52,60){\scriptsize A}
\end{picture}\] and $S_{\triangle B_1AB_4}=S_{\triangle B_2B_3A}$.
\end{ex} \pn We will work with determinants
$\Delta_1,\ldots,\Delta_n$,
where $$\Delta_i=(v_i,[v_{i+1},v_{i+2}])=
\begin{vmatrix}x_i&y_i&z_i\\\,x_{i+1}&y_{i+1}&z_{i+1}\\
\,x_{i+2}&y_{i+2}&z_{i+2}\end{vmatrix},\, i=1,\ldots,n-2,$$
$$\Delta_{n-1}=(v_{n-1},[v_n,v_1])=\begin{vmatrix}
\,x_{n-1}&y_{n-1}&z_{n-1}\\x_n&y_n&z_n\\x_1&y_1&z_1\end{vmatrix}
\text{ and } \Delta_n=(v_n,[v_1,v_2])=\begin{vmatrix}\,x_n&y_n&z_n\\
\,x_1&y_1&z_1\\\,x_2&y_2&z_2\end{vmatrix}$$ ($(x_j,y_j,z_j)$ are
coordinates of the vector $v_j$). It must be noted, that as our
polygon is generic, then these determinants are nonzero. \pmn Here
is the summary of obtained results.
\begin{itemize} \item A generic polygon $P=A_1\ldots A_n$ for even $n$
is regular, if $\Delta_1\cdot\Delta_3\cdot\ldots\cdot\Delta_{n-1}=
\Delta_2\cdot\Delta_4\cdot\ldots\cdot\Delta_n$. In this case its
support systems constitute an infinite family. If $n$ is odd, then
$P$ is regular, if
$\Delta_1\cdot\Delta_2\cdot\ldots\cdot\Delta_n>0$, then its
support system is unique up to the sign (Theorem 2.1). \item If
$n=4$ then $P$ is always regular. Each its derivative is a plane
self-intersecting quadrangle with oriented area $0$ (Theorem 3.1).
\item If $n=5$, then the derivative of a regular polygon is a
plane pentagon with oriented area $0$ (Theorem 4.1). \item A
regular hexagon is called \emph{strongly-regular}, if $\Delta_1=
\Delta_4$, $\Delta_2=\Delta_5$ and $\Delta_3=\Delta_6$. The
\emph{type} of a strongly-regular hexagon is the cyclic ratio
$\Delta_1:\Delta_2:\Delta_3$. All derivatives of a regular hexagon
$P$ are strongly-regular and of the same type. The types of
derived hexagon $P'$ and its derivative $P''$ are the same. If
$P'=B_1B_2B_3B_4B_5B_6$, then vertices of $P'$ belong to two
parallel planes $\Pi_1$ and $\Pi_2$: points $B_1$, $B_3$ and $B_5$
belong to $\Pi_1$, and points $B_2$, $B_4$ and $B_6$ belong to
$\Pi_2$. Moreover, the oriented area of the plane hexagon
$B_1B'_2B_3B'_4B_5B'_6$ is $0$ (here points $B'_2$, $B'_4$ and
$B'_6$ are projections of points $B_2,B_4,B_4$ to the plane
$\Pi_1$). (Theorems 5.1, 5.2 and 5.3). \end{itemize} In the last
section the question of the regularity of a knotted $6$-gon is
discussed.

\section{General remarks}
\pn We use notations from the previous section. \pmn
\begin{theor} Let $P$ be a generic $n$-gon in three-dimensional space. If
$n=2m$ is even, then $P$ is regular if and only if
$$\Delta_1\cdot\Delta_3\cdot\ldots\cdot\Delta_{2m-1}=
\Delta_2\cdot\Delta_4\cdot\ldots\cdot\Delta_{2m}.\eqno(1)$$ In this case
there is an infinite family of its support systems. If $n$ is odd then $P$
is regular if and only if $\Delta_1\cdot\Delta_2\cdot\ldots\cdot\Delta_n>0$.
In this case its support system is unique up to the sign. \end{theor}
\begin{proof} To prove the regularity we must construct a support
system. As $[u_{i-1},u_i]=v_i$ and $[u_i,u_{i+1}]=v_{i+1}$, then
$u_i$ is orthogonal to $v_i$ and $v_{i+1}$, hence, it is proportional to
their cross product: $u_i=\alpha [v_i,v_{i+1}]$. \pmn Let $u_1=[v_1,v_2]$.
As
$$[[v_i,v_{i+1}],[v_{i+1},v_{i+2}]]=\Delta_i\cdot v_{i+1}, \eqno(2)$$
then $u_2$ must be equal to $[v_2,v_3]/\Delta_1$. Indeed,
$$[u_1,u_2]=\left[[v_1,v_2],\dfrac{1}{\Delta_1}\cdot [v_2,v_3]\right]=
\dfrac{1}{\Delta_1}\cdot [[v_1,v_2],[v_2,v_3]]=v_2.$$ Analogously,
$$u_3=\dfrac{\Delta_1\cdot [v_3,v_4]}{\Delta_2},\,\,
u_4=\dfrac{\Delta_2\cdot
[v_4,v_5]}{\Delta_1\cdot\Delta_3},\,\,\ldots$$ Let $n=2m$ be even,
then
$$u_{2m}=\dfrac{\Delta_2\cdot\Delta_4\cdot\ldots\cdot\Delta_{2m-2}\cdot
[v_n,v_1]}{\Delta_1\cdot\Delta_3\cdot\ldots\cdot\Delta_{2m-1}}.$$
This choice of vectors $u_1,\ldots,u_n$ allows one to satisfy
conditions $[u_1,u_2]=v_2, \ldots,[u_{n-1},u_n]=v_n$. But the
condition $[u_n,u_1]=v_1$ can be satisfied only when
$$u_1=\dfrac{\Delta_1\cdot\Delta_3\cdot\ldots\cdot\Delta_{2m-1}\cdot
[v_1,v_2]}{\Delta_2\cdot\Delta_4\cdot\ldots\cdot\Delta_{2m}}.$$
But $u_1=[v_1,v_2]$, hence,
$$\Delta_1\cdot\Delta_3\cdot\ldots\cdot\Delta_{n-1}=\Delta_2\cdot
\Delta_4\cdot\ldots\cdot\Delta_n.$$ \pmn Let $n=2m$ and
$u_1,\ldots,u_n$ be a constructed above support system. Then
vectors $u'_1,\ldots,u'_n$, where $u'_i=\alpha u_i$ for even $i$,
and $u'_i=u_i/\alpha$ for odd $i$ also constitute a support system
for all nonzero $\alpha$. Hence, there is an infinite family of
support systems (and of derived polygons) for a regular $2m$-gon.
\pmn Let now $n$ be odd. We construct vectors $u_1,\ldots,u_n$, as
above. The key moment is the satisfaction of the condition
$$[u_n,u_1]=v_1\Leftrightarrow \dfrac{\Delta_1\cdot\Delta_3\cdot\ldots
\cdot\Delta_{n-2}\cdot
[[v_n,v_1],[v_1,v_2]]}{\Delta_2\cdot\Delta_4\cdot
\ldots\cdot\Delta_{n-1}}=v_1.$$ Now (2) gives us a sufficient
condition: $P$ is regular if
$\Delta_1\cdot\Delta_3\cdot\ldots\cdot\Delta_n=
\Delta_2\cdot\Delta_4\cdot\ldots\cdot\Delta_{n-1}$. However, in
the odd case this condition is not necessary. \pmn As above we
define vectors $u'_i$: $u'_i=\alpha\cdot u_i$, if $i$ is even, and
$u'_i=u_i/\alpha$, if $i$ is odd. Then $[u'_i,u'_{i+1}]=v_{i+1}$
for $i=1,\ldots,n-1$. But
$$[u'_n,u'_1]=\dfrac{\Delta_1\cdot\Delta_3\cdot\ldots\cdot\Delta_n}
{\alpha^2\cdot\Delta_2\cdot\Delta_4\cdot\ldots\cdot\Delta_{n-1}}\cdot
v_1.$$ Thus, if $\Delta_1\cdot\Delta_2\cdot\ldots\cdot\Delta_n>0$, then
there exists the unique (up to the sign) number $\alpha$ such, that
vectors $u'_1,\ldots,u'_n$ constitute a support system of $P$.
\end{proof}  \begin{rem} Let $n$ be odd. If an $n$-gon $P$ is not regular,
then its mirror-symmetric is regular. \end{rem} \begin{ex} Let us
consider a pentagon $P$:
$$\begin{array}{llll} v_1=(1,0,0)&\Delta_1=1&[v_1,v_2]=(0,0,1)&
u_1=(0,0,1)\\ v_2=(0,1,0)&\Delta_2=2&[v_2,v_3]=(1,0,0)& u_2=(1,0,0)\\
v_3=(0,0,1)&\Delta_3=-4&[v_3,v_4]=[2,2,0]&u_3=(1,1,0)\\
v_4=(2,-2,3)& \Delta_4=5&[v_4,v_5]=(5,-1,-4)&u_4=(-5/2,1/2,2)\\
v_5=(-3,1,-4)&
\Delta_5=-4&[v_5,v_1]=(0,-4,-1)&u_5=(0,8/5,2/5)\end{array}$$ As
$\Delta_1\cdot\Delta_2\cdot\Delta_3\cdot\Delta_4\cdot\Delta_5>0$, then
$P$ is regular, and as $[u_5,u_1]=(8/5,0,0)$, then $\alpha=2\sqrt 2/\sqrt 5$.
Thus, we have the following support system:
\begin{gather*}u'_1=\left(0,\,0,\,\dfrac{\sqrt 5}{2\sqrt 2}\right),\,
u'_2=\left(\dfrac{2\sqrt 2}{\sqrt 5},\,0,\,0\right),\,
u'_3=\left(\dfrac{\sqrt 5}{2\sqrt 2},\,\dfrac{\sqrt 5}{2\sqrt
2},\,0\right), \\ \\ u'_4=\left(-\sqrt{10},\,\dfrac{\sqrt 2}{\sqrt
5},\,\dfrac{4\sqrt 2}{\sqrt 5}\right),\,
u'_5=\left(0,\,\dfrac{2\sqrt 2}{\sqrt 5},\,\dfrac{1}{\sqrt
10}\right).\end{gather*} \begin{rem} Let $Q$ be an $n$-gon defined by
vectors $v_1,\ldots,v_n$, $v_1+\ldots+v_n=0$. $Q$ is a derived polygon
of some polygon $P$ if there exists a vector $u$ such, that
$$[u,u+v_1]+[u+v_1,u+v_1+v_2]+\ldots+[u+v_1+\ldots+v_{n-2},u+v_1+\ldots+
v_{n-1}]+[u+v_1+\ldots+v_{n-1},u]=0.$$ The lefthand part of this condition
is equal to $\sum_{0<i<j<n} [v_i,v_j]$ and does not depend on $u$ at all.
\end{rem}
\end{ex}

\section{Quadrangles}
\pn \begin{theor} Each quadrangle is regular and each its derivative is
a plane quadrangle with oriented area $0$. \end{theor}
\begin{proof} We use notations of the previous section. Let $P$ be a
quadrangle and $v_i=(a_i,b_i,c_i)$, $i=1,2,3,4$. Then
$$\Delta_2=\begin{vmatrix}\,a_2&b_2&c_2\\\,a_3&b_3&c_3\\ \,a_4&b_4&c_4
\end{vmatrix}=\begin{vmatrix} a_2&b_2&c_2\\a_3&b_3&c_3\\ -a_1-a_2-a_3&
-b_1-b_2-b_3&-c_1-c_2-c_3\end{vmatrix}=\begin{vmatrix} a_2&b_2&c_2\\
a_3&b_3&c_3\\-a_1&-b_1&-c_1\end{vmatrix}=-\Delta_1.$$ Analogously,
$\Delta_3=\Delta_1$ and $\Delta_4=-\Delta_1$, i.e. P is regular.
\pmn Let us consider a support system $u'_1,u'_2,u'_3,u'_4$, where
$u'_1=u_1/\alpha$, $u'_2=\alpha\cdot u_2$, $u'_3=u_3/\alpha$,
$u'_4=\alpha\cdot u_4$ and $u_1=[v_1,v_2]$,
$u_2=[v_2,v_3]/\Delta_1$, $u_3=-[v_3,v_4]$,
$u_4=-[v_4,v_1]/\Delta_1$. We must prove that the mixed product of
vectors $u'_2-u'_1$, $u'_3-u'_1$ and $u'_4-u'_1$ is zero. We have
\begin{multline*} (u'_2-u'_1,[u'_3-u'_1,u'_4-u'_1])=
(u'_2-u'_1,[u'_3,u'_4]-[u'_1,u'_4]-[u'_3,u'_1])=
(u'_2,[u'_3,u'_4])-(u'_2,[u'_1,u'_4])-\\-(u'_2,[u'_3,u'_1])-
(u'_1,[u'_3,u'_4])=\alpha\cdot(u_2,v_4)+\alpha\cdot(u_2,v_1)-
(u_1,v_3)/\alpha-(u_1,v_4)/\alpha=\\=\dfrac{\alpha}{\Delta_1}\cdot
\bigl(([v_2,v_3],v_4)+([v_2,v_3],v_1)\bigr)-\frac 1\alpha\cdot
\bigl(([v_1,v_2],v_3)+([v_1,v_2],v_4)\bigr)=\\=
\dfrac{\alpha}{\Delta_1}\cdot(\Delta_2+\Delta_1)-\frac
1\alpha\cdot (\Delta_1+\Delta_4)=0. \end{multline*}  Let
$u'_i=\overline{OB_i}$, $i=1,2,3,4$, and $\Pi$ be the plane of
points $B_1,B_2,B_3,B_4$. In a coordinate system, where $\Pi$ is
parallel to the plane $xy$, the $z$-coordinate of the sum
$[u'_1,u'_2]+[u'_2,u'_3]+[u'_3,u'_4]+[u'_4,u'_1]$ (which is zero)
is the oriented area of the quadrangle $B_1B_2B_3B_4$, multiplied
by $2$. In particular $B_1B_2B_3B_4$ is a self-intersecting
quadrangle.
\end{proof}

\section{Pentagons}
\pn \begin{theor} Let $P$ be a generic regular pentagon. Then its derivative
is a plane pentagon with oriented area zero. \end{theor}  \begin{proof}
Let $u_1,\ldots,u_5$ be a support system for $P$. We must prove that
$(u_2-u_1,[u_3-u_1,u_4-u_1])=0$ and $(u_3-u_1,[u_4-u_1,u_5-u_1])=0$. As
$[u_1,u_2]+[u_2,u_3]+[u_3,u_4]+[u_4,u_5]+[u_5,u_1]=0$, then
\begin{multline*} (u_2-u_1,[u_3-u_1,u_4-u_1])=(u_2,[u_3,u_4])-
(u_2,[u_3,u_1])-(u_2,[u_1,u_4])-(u_1,[u_3,u_4])=\\=(u_4,[u_2,u_3])+
(u_4,[u_1,u_2])-(u_1,[u_2,u_3])-(u_1,[u_3,u_4])=-(u_4,[u_5,u_1])+
(u_1,[u_4,u_5])=0.\end{multline*} The second equality can be proved
analogously. \pmn The second statement can be proved in the same way, as
the analogous statement in Theorem 3.1. \end{proof}
\begin{rem} If vectors $u_1,u_2,u_3,u_4,u_5$ have the same $z$-coordinate,
then $x$- and $y$-coordinates of the sum
$[u_1,u_2]+[u_2,u_3]+[u_3,u_4]+ [u_4,u_5]+[u_5,u_1]$ are zero.
\end{rem} \begin{rem} Up to rotation, dilation and mirror symmetry
each generic pentagon can be recovered from a pentagon in the
plane $z=1$, with oriented area zero and with centroid positioned
at positive $x$ half-axis. \end{rem} \begin{ex} Let
$$u_1=(2,2,1),\,\,u_2=(3,-1,1),\,\,u_3=(-3,1,1),\,\,u_4=(-4,0,1),\,\,
u_5=(-1,-1,1).$$ Endpoints $A,B,C,D,E$ of these vectors belong to
the plane $z=1$ and define at this plane the self-intersecting
pentagon with oriented area zero:
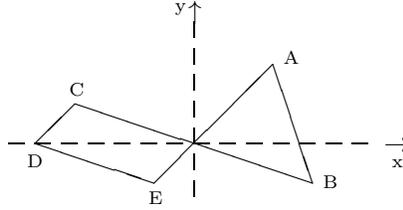
\[\begin{picture}(150,75)
\multiput(0,25)(10,0){14}{\line(1,0){6}}
\multiput(70,5)(0,10){7}{\line(0,1){6}} \put(145,16){\scriptsize
x} \put(63,75){\scriptsize y} \put(142,22){$\to$}
\put(68,72){$\uparrow$} \put(55,10){\line(1,1){45}}
\put(100,55){\line(1,-3){15}} \put(25,40){\line(3,-1){90}}
\put(10,25){\line(1,1){15}} \put(10,25){\line(3,-1){45}}
\put(104,55){\scriptsize A} \put(119,8){\scriptsize B}
\put(23,43){\scriptsize C} \put(7,16){\scriptsize D}
\put(53,2){\scriptsize E} \end{picture}\] Thus, vectors
$u_1,\ldots,u_5$ constitute as a support system:
\begin{gather*}v_2=[u_1,u_2]=(3,1,-8),\,v_3=[u_2,u_3]=(-2,-6,0),\,v_4=[u_3,u_4]=
(1,-1,4),\\v_5=[u_4,u_5]=(1,3,4),\,v_1=[u_5,u_1]=(-3,3,0).\end{gather*}
Now let us perform the inverse computation:
$$\Delta_1=192,\,\Delta_2=-128,\,\Delta_3=32,\,\Delta_4=48,\,
\Delta_5=-144$$ and
\begin{align*} &w_1=[v_1,v_2]=(-24,-24,-12)\\
&w_2=\dfrac{[v_2,v_3]}{\Delta_1}=\left(-\frac 14,\frac{1}{12},
-\frac{1}{12}\right)\\ &w_3=\dfrac{\Delta_1\cdot
[v_3,v_4]}{\Delta_2} =(36,-12,-12)\\ &w_4=\dfrac{\Delta_2\cdot
[v_4,v_5]}
{\Delta_1\cdot\Delta_3}=\left(\frac 13,0,-\frac{1}{12}\right)\\
&w_5=\dfrac{\Delta_1\cdot\Delta_3\cdot [v_4,v_5]}
{\Delta_2\cdot\Delta_4}=(12,12,-12)\end{align*}  As
$[w_5,w_1]=144\cdot v_1$, then $\alpha=12$ and we return to the initial
set $u_1,\ldots,u_5$. \end{ex}

\section{Hexagons}
\pn \begin{ex} Let vectors $v_1=(1,0,0)$, $v_2=(0,1,0)$, $v_3=(0,0,1)$,
$v_4=(2,-1,3)$, $v_5=(-1,5,2)$ and $v_6=(-2,-5,-6)$ define the hexagon $P$.
As
$$\Delta_1=1,\,\Delta_2=2,\,\Delta_3=9,\,\Delta_4=15,\,\Delta_5=20,\,
\Delta_6=-6$$ and
$\Delta_1\Delta_3\Delta_5=\Delta_2\Delta_4\Delta_6= -180$, then
$P$ is regular. Vectors
$$u_1=(0,0,1),\,u_2=(1,0,0),\,u_3=\left(\frac 12,1,0\right),\,
u_4=\left(-\frac{34}{9},-\frac{14}{9},2\right),\,
u_5=\left(-6,-3,\frac 92\right),\,u_6=\left(0,1,-\frac 56\right)$$
constitute a support system, and vectors
\begin{gather*}u_2-u_1=(1,0,-1),\,u_3-u_2=\left(-\frac
12,1,0\right),\,u_4-u_3=
\left(-\frac{77}{18},-\frac{-23}{9},2\right),\,u_5-u_4=
\left(-\frac{20}{9},-\frac{13}{9},\frac 52\right),\\u_6-u_5=
\left(6,4,-\frac{16}{3}\right),\,u_1-u_6=\left(0,-1,\frac{11}{6}
\right)\end{gather*} define the derived hexagon $P'$. The
corresponding determinants are:
$$\Delta'_1=-\frac{32}{9},\,\Delta'_2=8,\,\Delta'_3=\frac 43,\,
\Delta'_4=-\frac{32}{9},\,\Delta'_5=8,\,\Delta'_6=\frac 43.$$ We see that
$P'$ satisfies a stronger condition, than regularity. \end{ex}
\begin{defin} A regular hexagon is called \emph{strongly-regular}, if
$\Delta_1=\Delta_4$, $\Delta_2=\Delta_5$ and $\Delta_3=\Delta_6$.
The \emph{type} of a strongly-regular hexagon is the cyclic ratio
$\Delta_1:\Delta_2:\Delta_3$. \end{defin} \begin{theor} All
derivatives of a regular hexagon $P$ are strongly-regular and have
the same type.
\end{theor} \begin{proof} We use notations from Section 2. A
derived hexagon $P'$ is defined by vectors
$u'_2-u'_1,u'_3-u'_2,\ldots,u'_6-u'_5, u'_1-u'_6$, where
$u'_i=\alpha\cdot u_i$ for even $i$, and $u'_i=u_i/\alpha$ for odd
$i$. We must prove that $\Delta'_1=\Delta'_4$,
$\Delta'_2=\Delta'_5$ and $\Delta'_3=\Delta'_6$ (determinants
$\Delta'_i$ are defined in the same way, as determinants
$\Delta_i$, but for polygon $P'$). We will prove the first
equality. It can be rewritten in the form
$$(\alpha\cdot u_2-u_1/\alpha,[u_3/\alpha-\alpha\cdot u_2,
\alpha\cdot u_4-u_3/\alpha])= (u_5/\alpha-\alpha\cdot
u_4,[\alpha\cdot u_6-u_5/\alpha, u_1/\alpha-\alpha\cdot u_6]).$$
Here coefficients at $\alpha^3$ and $\alpha^{-3}$ are zero. Let
compare coefficients at $\alpha$ and at $\alpha^{-1}$ in the left
and in the righthand sides of this relation. Coefficient at
$\alpha^{-1}$ in the left is
$$-(u_1,[u_3,u_4])-(u_1,[u_2,u_3])=-(u_1,v_4)-(u_1,v_3)=
-([v_1,v_2],v_4)-([v_1,v_2],v_3)=-([v_1,v_2],v_4)-\Delta_1.$$
Coefficient at $\alpha^{-1}$ in the right is
\begin{multline*} (u_5,[u_6,u_1])+(u_4,[u_5,u_1])=(u_1,[u_5,u_6])+
(u_1,[u_4,u_5])=(u_1,v_6)+(u_1,v_5)=\\=([v_1,v_2],v_6)+([v_1,v_2],v_5)=
\Delta_6-([v_1,v_2],v_3+v_4+v_6)=-([v_1,v_2],v_4)-\Delta_1.\end{multline*}
Then, coefficient at $\alpha$ in the left is
\begin{multline*} (u_2,[u_3,u_4])+(u_1,[u_2,u_4])=(u_2,v_4)+(u_4,[u_1,u_2])=
(u_2,v_4)+(u_4,v_2)=\\=\dfrac{1}{\Delta_1}\cdot([v_2,v_3],v_4)+
\dfrac{\Delta_2}{\Delta_1\Delta_3}\cdot([v_4,v_5],v_2)=
\dfrac{\Delta_2}{\Delta_1}-\dfrac{\Delta_2}{\Delta_1\Delta_3}\cdot
([v_4,v_5],v_1+v_3+v_6)=\\=-\dfrac{\Delta_2}{\Delta_1\Delta_3}
([v_4,v_5],v_1)-\dfrac{\Delta_2\Delta_4}{\Delta_1\Delta_3}\,.\end{multline*}
Coefficient at $\alpha$ in the right is
\begin{multline*}-(u_4,[u_6,u_1])-(u_4,[u_5,u_6])=
-(u_4,v_1)-(u_4,v_6)=
-\dfrac{\Delta_2}{\Delta_1\Delta_3}\cdot([v_4,v_5],v_1)-\\-
\dfrac{\Delta_2}{\Delta_1\Delta_3}\cdot([v_4,v_5],v_6)=
-\dfrac{\Delta_2}{\Delta_1\Delta_3}\cdot([v_4,v_5],v_1)-
\dfrac{\Delta_2\Delta_4}{\Delta_1\Delta_3}\,. \end{multline*}
Conditions $\Delta'_2=\Delta'_5$ and $\Delta'_3=\Delta'_6$ can be proved
analogously. \pmn Now let us turn to the second part of the theorem. We
need to prove that ratios
$$\frac{(\alpha\cdot u_2-u_1/\alpha,[u_3/\alpha-\alpha\cdot u_2,
\alpha\cdot u_4-u_3/\alpha])} {(u_3/\alpha-\alpha\cdot
u_2,[\alpha\cdot u_4-u_3/\alpha, u_5/\alpha-\alpha\cdot u_4])}\,
\text{ and }\, \frac{(u_3/\alpha-\alpha\cdot u_2,[\alpha\cdot
u_4-u_3/\alpha, u_5/\alpha-\alpha\cdot u_4])} {(\alpha\cdot
u_4-u_3/\alpha,[u_5/\alpha-\alpha\cdot u_4, \alpha\cdot
u_6-u_5/\alpha])}$$ are constants, as functions of $\alpha$. We
will prove it for the first ratio, which can be rewritten in the
following way:
$$\dfrac{\alpha\cdot[(u_2,v_4)+(u_4,v_2)]-
\frac 1\alpha\cdot[(u_1,v_4)+(u_1,v_3)]} {\frac
1\alpha\cdot[(u_3,v_5)+(u_5,v_3)]-\alpha\cdot[(u_2,v_5)+(u_2,v_4)]}.$$
It is enough to prove that the ratio of coefficients at $\alpha$
and the ratio of coefficients at $\alpha^{-1}$ are the same, i.e.
to prove that
\begin{multline*}\frac{(u_2,v_4)+(u_4,v_2)}{(u_2,v_5)+(u_2,v_4)}=
\frac{(u_1,v_3)+(u_1,v_4)}{(u_3,v_5)+(u_5,v_3)} \Leftrightarrow \\
\Leftrightarrow
\dfrac{\dfrac{\Delta_2}{\Delta_1}+\dfrac{\Delta_2}{\Delta_1\Delta_3}
\cdot(v_2,[v_4,v_5])}{\dfrac{1}{\Delta_1}\cdot (v_2,[v_3,v_5])+
\dfrac{\Delta_2}{\Delta_1}}=\dfrac{\Delta_1+(v_1,[v_2,v_4])}
{\dfrac{\Delta_1\Delta_3}{\Delta_2}+\dfrac{\Delta_1\Delta_3}
{\Delta_2\Delta_4}\cdot (v_3,[v_5,v_6])} \Leftrightarrow\\
\Leftrightarrow
\dfrac{\Delta_3+(v_2,[v_4,v_5])}{\Delta_2+(v_2,[v_3,v_5])}=
\dfrac{\Delta_4\cdot(\Delta_1+(v_1,[v_2,v_4])}{\Delta_1\cdot(\Delta_4+
(v_3,[v_5,v_6])}\end{multline*}  Let
$$(a_1,a_2,a_3),(b_1,b_2,b_3),(c_1,c_2,c_3),(d_1,d_2,d_3),(e_1,e_2,e_3),
(f_1,f_2,f_3)$$ be coordinates of vectors
$u_1,u_2,u_3,u_4,u_5,u_6$, respectively. And let $M$ be the matrix
of coordinates of vectors $v_1,\ldots,v_6$:
$$M=\begin{pmatrix} f_2a_3-f_3a_2&f_3a_1-f_1a_3&f_1a_2-f_2a_1\\
a_2b_3-a_3b_2&a_3b_1-a_1b_3&a_1b_2-a_2b_1\\ b_2c_3-b_3c_2&
b_3c_1-b_1c_3&b_1c_2-b_2c_1\\ c_2d_3-c_3d_2&c_3d_1-c_1d_3&
c_1d_2-c_2d_1\\ d_2e_3-d_3e_2&d_3e_1-d_1e_3&d_1e_2-d_2e_1\\
e_2f_3-e_3f_2&e_3f_1-e_1f_3&e_1f_2-e_2f_1\end{pmatrix}$$ Let us
denote by $\Delta_{ijk}$ the third order determinant of the
submatrix, composed of rows of $M$ with numbers $i$, $j$ and $k$.
Thus, in these notation, $\Delta_1$ and $\Delta_{123}$ are the
same, $\Delta_2$ and $\Delta_{234}$ are the same, and so on. The
equality, we need to prove, can be rewritten, as:
$$\Delta_1\cdot(\Delta_3+\Delta_{245})\cdot(\Delta_4+\Delta_{356})=
\Delta_4\cdot(\Delta_1+\Delta_{124})\cdot(\Delta_2+\Delta_{235}).$$
Each factor in this equality, as a polynomial in variables
$a_i,b_i,\ldots,f_i$, is a product of two irreducible polynomials.
Thus, we have the product of six polynomials in the lefthand side
of the equality, and the product of six polynomials in the
righthand side. However, four of them can be cancelled. Now, the
lefthand side is the product of two polynomials $p_1p_2$, where
$p_1$ depends only on $a_i,b_i,c_i,d_i$ and $p_2$ --- on
$b_i,c_i,d_i,e_i$. The righthand side is the product of two
polynomials $q_1q_2$ with the same properties. \pmn An easy
computation demonstrates, that for any choice of vectors
$u_1,\ldots,u_6$ the equality $\Delta_1\Delta_3\Delta_5=
\Delta_2\Delta_4\Delta_6$ is automatically satisfied. But these
vectors constitute a support system only if the sum of matrix $M$
rows is zero, i.e. if the following conditions are satisfied
$$\begin{array}{c}
f_2a_3-f_3a_2+a_2b_3-a_3b_2+b_2c_3-b_3c_2+c_2d_3-c_3d_2+d_2e_3-d_3e_2+
e_2f_3-e_3f_2=0\\
f_3a_1-f_1a_3+a_3b_1-a_1b_3+b_3c_1-b_1c_3+c_3d_1-c_1d_3+d_3e_1-d_1e_3+
e_3f_1-e_1f_3=0\\
f_1a_2-f_2a_1+a_1b_2-a_2b_1+b_1c_2-b_2c_1+c_1d_2-c_2d_1+d_1e_2-d_2e_1+
e_1f_2-e_2f_1=0\end{array}\eqno(3)$$  If we eliminate from these
three relations variables $f_3$ and $f_2$, then variable $f_1$
will be eliminated also and we will obtain the equation $r=0$,
where $r$ is a polynomial that depends on variables
$a_i,b_i,c_i,d_i,e_i$.\pmn Let us transform the relation
$p_1p_2=q_1q_2$ into relation $\tilde{p}_1p_2=\tilde{q_1}q_2$,
where
$$\tilde{p}_1=\text{resultant}(p_1,r,a_3),\,\tilde{q}_1=
\text{resultant}(q_1,r,a_3).$$ Polynomials $\tilde{p}_1$ and
$\tilde{q}_1$ each are products of two irreducible polynomials. In
result, three factors in the lefthand side are cancelled with
three factors in the righthand side.\end{proof}
\begin{theor} Let $P$ be a regular hexagon, $P'$ be its derivative and
$P''$ be a derivative of $P'$. Then types of $P'$ and $P''$ are the same.
\end{theor} \begin{proof} Let vectors $u_1=(a_1,a_2,a_3),\ldots,u_6=(f_1,f_2,f_3)$
constitute a support system of hexagon $P$ and vectors
$w_1,\ldots,w_6$ constitute a support system of hexagon $P'$. We
will work with the matrix $M'$ of coordinates of vectors
$u_2-u_1,\ldots,u_1-u_6$ and with the matrix $M''$ of coordinates
of vectors $w_2-w_1,\ldots,w_1-w_6$. Let $\Delta'_i$ and
$\Delta''_i$ be corresponding determinants. Let us prove, that
$$\dfrac{\Delta''_1}{\Delta''_2}=\dfrac{\Delta'_2}{\Delta'_3},\quad
\dfrac{\Delta''_2}{\Delta''_3}=\dfrac{\Delta'_3}{\Delta'_1},\quad
\dfrac{\Delta''_3}{\Delta''_1}=\dfrac{\Delta'_1}{\Delta'_2}.$$ We
will prove the first equality. As
$$\Delta''_1/\Delta''_2=
\dfrac{(w_2-w_1,[w_3-w_2,w_4-w_3])}{(w_3-w_2,[w_4-w_3,w_5-w_4])}=
\dfrac{\Delta'_2/\Delta'_1-\Delta'_{124}+\Delta'_2\Delta'_{245}/
\Delta'_1\Delta'_3
-\Delta'_1}{\Delta'_1\Delta'_3/\Delta'_2-\Delta'_{235}/\Delta'_1+
\Delta'_1\Delta'_3\Delta'_{356}/\Delta'_2\Delta'_4-
\Delta'_2/\Delta'_1}\,,$$ then
\begin{multline*}
\Delta''_1/\Delta''_2=\Delta'_2/\Delta'_3 \Leftrightarrow\\
\Leftrightarrow\Delta'_2\Delta'_3/\Delta'_1-\Delta'_3\Delta'_{124}+
\Delta_2\Delta_{245}/\Delta_1-\Delta_1\Delta_3=
\Delta'_1\Delta'_3-\Delta'_2\Delta'_{235}/\Delta'_1+
\Delta'_3\Delta'_{356}-\left(\Delta'_2\right)^2/\Delta'_1.\end{multline*}
(We remind that $\Delta'_1=\Delta'_4$.) Now we separate summands
of the first degree in $\Delta'$ from summands of the second
degree in $\Delta'$ and will prove that
$$\begin{array}{l}\Delta'_2+\Delta'_3+\Delta'_{235}+\Delta'_{245}=0\\
2\Delta'_1+\Delta_{124}+\Delta'_{356}=0\end{array}\eqno(4)$$ Let
us remind that variables $a_i,b_i,\ldots,f_i$ satisfy conditions
(3). \pmn The lefthand side of the first equation in (4) is a
polynomial $g$ of degree 3 in variables $b_i,c_i,d_i,e_i,f_i$ and
this polynoimal is exactly the result of the elimination of
variables $a_i$ from relations (3). \pmn The lefthand side of the
second equation in (4) is a polynomial $h$ in all variables
$a_i,\ldots,f_i$. If we eliminate from $h$ variables $a_i$, using
relations (3), then we will obtain $g$. \end{proof}
\begin{ex} If a hexagon $P$ is strongly-regular and $P'$ is the
derived hexagon, then they can be of different types. Here is the
example: let $P$ be defined by vectors
$$v_1=(1,0,0),\,v_2=(0,1,0),\,v_3=(0,0,1),\,v_4=(2,-1,3),\,
v_5=\Bigl(-\frac 32,-\frac 12,-\frac 32\Bigr),\,v_6=\Bigl(-\frac
32, \frac 12,-\frac 52\Bigr).$$ Then $\Delta_1=1$,
$\Delta_2=2$, $\Delta_3=-\frac 52$, $\Delta_4=1$, $\Delta_5=2$,
$\Delta_6=-\frac 52$, i.e. $P$ is strongly-regular. Then
$$u_1=(0,0,1),\,u_2=(1,0,0),\,u_3=\Bigl(\frac 12,1,0\Bigr),\, u_4=
\Bigl(-\frac{12}{5},\frac 65,2\Bigr),\,u_5=\Bigl(-\frac
52,\frac{15}{8}, \frac{15}{8}\Bigr),\,u_6=\Bigl(0,1,\frac
15\Bigr).$$ The derived hexagon is defined by vectors
$$\begin{array}{lll} u_2-u_1=(1,0,-1)&u_3-u_2=
\left(-\dfrac 12,1,0\right)&u_4-u_3=\left(-\dfrac{29}{10},\dfrac
15,2\right)\\
\\u_5-u_4=\left(-\dfrac{1}{10},\dfrac{27}{40},-\dfrac 18\right)&
u_6-u_5=\left(\dfrac 52,-\dfrac 78,-\dfrac{67}{40}\right)&
u_1-u_6=\left(0,-1,\dfrac 45\right)\end{array}$$ The corresponding
determinants are as follows.
$$\Delta'_1=-\dfrac 45,\,\, \Delta'_2=\dfrac 18,\,\,
\Delta'_3=\dfrac{3}{10},\,\, \Delta'_4=-\dfrac 45,\,\,
\Delta'_5=\dfrac 18,\,\, \Delta'_6=\dfrac{3}{10}\,.$$ We see that
types of $P$ and $P'$ are different. The reason here is that $P$
is not a derivative of some regular hexagon, because
$\sum_{0<i<j<6}[v_i,v_j]\neq 0$ (see Remark 2.2).
\end{ex}

\section{The geometry of a derived hexagon}
\pn Vectors $u_1=(a_1,a_2,a_3),\ldots,u_6=(f_1,f_2,f_3)$ constitute
a support system, if conditions (3) are satisfied. The elimination
of variables $f_1,f_2,f_3$ from these conditions gives us the following
relation:
\begin{multline*}
(a_1b_2c_3-a_1b_3c_2+a_2b_3c_1-a_2b_1c_3+a_3b_1c_2-a_3b_2c_1)-\\
-(a_1b_2e_3-a_1b_3e_2-a_2b_1e_3+a_2b_3e_1+a_3b_1e_2-a_3b_2e_1)+\\
+(a_1c_2d_3-a_1c_3d_2+a_2c_3d_1-a_2c_1d_3+a_3c_1d_2-a_3c_2d_1)+\\
+(a_1d_2e_3-a_1d_3e_2+a_2d_3e_1-a_2d_1e_3+a_3d_1e_2-a_3d_2e_1)-\\
-(b_1c_2e_3-b_1c_3e_2+b_2c_3e_1-b_2c_1e_3+b_3c_1e_2-b_3c_2e_1)-\\
-(c_1d_2e_3-c_1d_3e_2+c_2d_3e_1-c_2d_1e_3+c_3d_1e_2-c_3d_2e_1)=0
\end{multline*}  The lefthand side of this relation is the, multiplied by $6$,
oriented volume of the following polyhedron $Q$: let
$B_1,B_2,B_3,B_4,B_5$ be endpoints of vectors
$u_1,u_2,u_3,u_4,u_5$, respectively. Then $Q$ is the result of
pasting together tetrahedrons $B_1B_3B_5B_2$ and $B_1B_3B_5B_4$
along the face $B_1B_3B_5$. But the volume of $Q$ is zero, hence,
vertices $B_2$ and $B_4$ are in one half-space with respect to the
plane $B_1B_3B_5$ and at the same distance from it. \pmn If we
start our numeration with the vector $u_3$, then in the same way
we will have that vertices $B_4$ and $B_6$ are in one half-space
with respect to $B_1B_3B_5$ and at the same distance from it.
Thus, we have the following configuration of points
$B_1,\ldots,B_6$: points $B_1,B_3,B_5$ belong to plane a $\Pi_1$
and points $B_2,B_4,B_6$ belong to a parallel plane $\Pi_2$. We
can assume that these planes are parallel to the plane $xy$. Let
$B'_2$, $B'_4$ and $B'_6$ be projections of points $B_2$, $B_4$
and $B_6$ to plane $\Pi_1$, then the third coordinate of the sum
$$[u_1,u_2]+[u_2,u_3]+[u_3,u_4]+[u_4,u_5]+[u_5,u_6]+[u_6,u_1]$$
(i.e. zero) is the oriented area of the plane hexagon
$B_1B'_2B_3B'_4B_5B'_6$. \pmn Thus, up to rotation and dilation, we have
the following construction of a regular hexagon: we draw a zero-area
hexagon $B_1B_2B_3B_4B_5B_6$ at the plane $xy$. Then vertices with
even numbers we lift to the plane $z=1$. Then we choose a point $O$ at
the $z$-axis: vectors $\overline{OB_1},\ldots,\overline{OB_6}$  are
vectors $u_1,\ldots,u_6$. It must be noted that in terms of this
construction $x$- and $y$-coordinates of the sum $[u_1,u_2]+\ldots+
[u_6,u_1]$ are zero automatically.

\section{Knots}
\pn \begin{prop} A pentagon in space cannot be knotted.\end{prop}\pmn
\emph{Sketch of proof.} Let a pentagon $ABCDE$ be knotted as a trefoil:
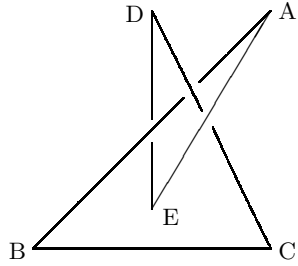
\[\begin{picture}(110,110) \put(10,10){\line(1,0){90}}
\qbezier(10,10)(40,40)(67,67) \qbezier(73,73)(85,85)(100,100)
\qbezier(55,25)(55,40)(55,50) \qbezier(55,59)(55,70)(55,100)
\qbezier(55,100)(70,70)(74,62) \qbezier(78,56)(85,40)(100,10)
\put(55,25){\line(3,5){45}} \put(103,97){\small A}
\put(1,6){\small B} \put(103,6){\small C} \put(45,96){\small D}
\put(59,19){\small E} \end{picture}\] Let vertices $A$, $B$ and $C$ belong
to the plane $xy$. Then the point $D$ is above this plane and the point
$E$ --- below. But then all segment $[EA]$ is below $xy$ and cannot pass
above the segment $CD$. \qed \pmn A hexagon in space can be knotted,
however:
\begin{prop} A knotted hexagon is not regular. \end{prop} \pmn
\emph{Sketch of proof.} Let $ABCDEF$ be a hexagon knotted as a trefoil:
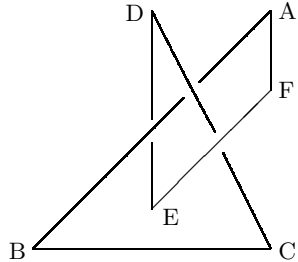
\[\begin{picture}(110,110) \put(10,10){\line(1,0){90}}
\qbezier(10,10)(40,40)(67,67) \qbezier(73,73)(85,85)(100,100)
\qbezier(55,25)(55,40)(55,50) \qbezier(55,59)(55,70)(55,100)
\qbezier(55,100)(70,70)(79,54) \qbezier(82,46)(85,40)(100,10)
\put(55,25){\line(1,1){45}} \put(100,70){\line(0,1){30}}
\put(103,97){\small A} \put(1,6){\small B} \put(103,6){\small C}
\put(45,96){\small D} \put(59,19){\small E} \put(103,67){\small F}
\end{picture}\] As above we will assume that points $A$, $B$ and $C$
belong to the plane $xy$. As $D$ is above $xy$, then the vector
$\overline{CD}$ is directed "up", thus, the triple
$\{\overline{AB},\overline{BC},\overline{CD}\}$ is right and
$\Delta_1>0$. \pmn The point $E$ is below $xy$, hence, from $E$'s
point of view, the rotation from the vector $\overline{BC}$ to the
vector $\overline{CD}$ is clockwise, i.e. the triple
$\{\overline{BC}, \overline{CD},\overline{DE}\}$ is left and
$\Delta_2<0$. \pmn As the point $F$ is above the plane $CDE$, then
the triple $\{\overline{CD}, \overline{DE},\overline{EF}\}$ is
right and $\Delta_3>0$. \pmn Points $D$ and $F$ are above $xy$ and
the point $E$ --- below, hence, from $A$'s point of view, the
rotation from the vector $\overline{DE}$ to the vector
$\overline{EF}$ is clockwise, i.e. the triple
$\{\overline{DE},\overline{EF},\overline{FA}\}$ is left and
$\Delta_4<0$. \pmn Analogously, the triple
$\{\overline{EF},\overline{FA},\overline{AB}\}$ is right and
$\Delta_5>0$. At last, the triple $\{\overline{FA},
\overline{AB},\overline{BC}\}$ is left and $\Delta_6<0$. \pmn We
have that $\Delta_1\Delta_3\Delta_5>0$ and
$\Delta_2\Delta_4\Delta_6<0$, i.e. the hexagon $ABCDEF$ is
irregular. \qed

\vspace{1cm}
\end{document}